\documentclass[a4paper,11pt]{amsart}
\usepackage{mathrsfs}
\usepackage{color}
\usepackage[dvipdfmx]{graphicx}

\title{Confidence disc and square for Cauchy distributions}

\author[Y. Akaoka]{Yuichi Akaoka}
\thanks{A part of this paper consists of Y.~A.'s master's thesis \cite{Akaoka2020a}.}
\address{Department of Mathematics, Faculty of Science, Shinshu University, 3--1--1, Asahi, Matsumoto, Nagano, 390--8621}
\curraddr{Gunma bank}
\email{18ss101b@gmail.com}

\author[K. Okamura]{Kazuki Okamura}
\address{Department of Mathematics, Faculty of Science, Shizuoka University, 836, Ohya, Suruga-ku, Shizuoka, 422--8527}
\email{okamura.kazuki@shizuoka.ac.jp}

\author[Y. Otobe]{Yoshiki Otobe}
\address{Department of Mathematics, Faculty of Science, Shinshu University, 3--1--1, Asahi, Matsumoto, Nagano, 390--8621}
\email{otobe@math.shinshu-u.ac.jp}

\date{\today}
\dedicatory{}

\theoremstyle{plain}
\newtheorem{theorem}{Theorem}[section]
\newtheorem{proposition}[theorem]{Proposition}
\newtheorem{lemma}[theorem]{Lemma}
\newtheorem{corollary}[theorem]{Corollary}
\theoremstyle{definition}

\theoremstyle{remark}
\newtheorem{remark}{Remark}[section]

\ifdim\overfullrule>0pt 
 
\fi

\def\E{{\mathbb{E}}}

\def\R{{\mathbb{R}}}
\def\C{{\mathbb{C}}}

\def\F{{\mathscr{F}}}
\def\Z{{\mathbb{Z}}}
\def\V{{\mathrm{Var}}}
\def\PV{{\mathrm{PV}}}
\DeclareMathOperator{\med}{med}
\def\e{\varepsilon}
\def\<{\left<}
\def\>{\right>}

\begin{document}
\begin{abstract}
We will construct a confidence region of parameters
for a sample of size $N$ from Cauchy distributed
random variables.
Although Cauchy distribution has two parameters,
a location parameter $\mu \in \R$ and a scale parameter $\sigma > 0$,
we will infer them at once by regarding them as a single complex
parameter $\gamma := \mu + i\sigma$.
The region should be a domain in the complex plane, and
we will give a simple and concrete formula to give the region
as a disc and a square.
\end{abstract}
\maketitle

\setcounter{tocdepth}{1}

\section{Introduction}
The Cauchy distribution is one of typical bell-shaped and stable laws
that has a longer and flatter tail,
which makes mathematical handling difficult.
But for that reason, typical samples from the Cauchy distribution
concentrate on its centre except for a few ``outliers'' so that
there might be cases that to model in Cauchy distribution rather than Gaussian
is preferable.
From this point of view,
in preceding articles \cite{Akaoka2021-2,Akaoka2021-1},
we have investigated estimators for the parameters of Cauchy distributions
from observations.
They are point estimations.
We are, however, in the present paper, concerned with an interval
estimation, that is, to construct confidence regions.
Confidence regions are not only practically often used
but also relate to statistical hypothesis testing
(see e.g., \cite[\S 7.1.2]{Shao2003}),
so it goes without saying its statistical importance.

On the other hand, to our knowledge, no concrete regions for
Cauchy distributions have been known, while few studies examined them.
One of them is a paper by
Haas, Bain and Antle \cite{Haas1970} which is based on \cite{Haas1969}.
They discussed mainly the maximal likelihood estimators for
Cauchy distributions and proposed confidence intervals
for the location parameter $\mu$;
but unfortunately, the explicit formula for the maximal likelihood estimator is known only for the case of sample sizes of 3 and 4, and furthermore, there is no algebraic closed-form formula for the case of sample size of 5.
See \cite{Ferguson1978} and \cite{Okamura2021}.
If the explicit formula is not known, then, the Newton--Raphson method has often been used.
Hinkley \cite[(2.12)]{Hinkley1978}
gives confidence regions for the pair of the location and scale by using the maximal likelihood estimator.
It is related to the likelihood ratio test, however, it is complicated,
and we are hard to imagine the concrete shape of the region.
Vrbik also obtained confidence regions \cite{Vrbik2013}
based on the maximal likelihood estimator.
See also \cite{CohenFreue2007, Kravchuk2012, Lawless1972}
for related inference methods for Cauchy distributions.

We also would like to point out that, though it is certainly true
those numerical simulations using a computer
have been becoming an effective
method for statistically testing or even constructing confidence intervals,
it is not trivial to apply them for Cauchy distributions.
To see it, we see, for example, a simple arithmetic mean of the Cauchy random variables
still obeys the Cauchy distribution which is
independent of the number of the sizes of samples:
we emphasize here that to choose a simple and nice estimator for
Cauchy distributions is non-trivial.
Let us also mention that few unbiased estimators were known so far
except, for instance, using MLEs \cite{McCullagh1993} or trimmed average \cite{Rothenberg1964} which actually eliminates the outliers from the observations.

In this situation, we would like to propose in the present paper
to take a geometric mean
(see Remark \ref{rem:geometric-mean} for the terminology).
There are several reasons why it is nicer than other quantities,
but since this problem belongs to the theory of point estimation,
we would not like to go in this direction further here.
We point out here only that the geometric mean has enough high integrability, and is unbiased.
See \cite{Akaoka2021-1} for some additional properties of the geometric mean.
Also since the geometric mean is an exponential of the arithmetic mean of the
logarithms, it is harmless to compute the quantity.
We, however, note that the geometric mean $\prod_{j=1}^N X_j^{1/N}$ is a
product of $1/N$-th power of (probably negative) observations.
So we are naturally led to treating all the quantities as complex numbers.
In particular, $X^p$ ($0 < |p| < 1$) and $\log X$ are all complex random variables.
Therefore, we will gather some basic quantities---mainly expectations of
such random variables---in Section \ref{sec:auxiliary-results}.

It is well known that, if $X$ obeys a Cauchy distribution,
$\E[|X|^p]$ exists so that $\E[X^p]$ also exists for $|p| < 1$.
The explicit formula of $\E[|X|^p]$  is given by \cite[(3.0.3)]{Zolotarev1986}.
However, the derivation of \cite[(3.0.3)]{Zolotarev1986} is very complicated, so we will give another way of calculations of them along with $\E[X^p]$  in Section \ref{sec:auxiliary-results}.
They are  rather simple and direct.
We also would like to expect that,
through these manipulations, it becomes clear
that considering in the complex plane is not a technical restriction
but an essential tool for the Cauchy distributions.
In particular, we hope it will be natural to regard the location parameter $\mu$
and the scale parameter $\sigma > 0$ of the Cauchy distribution
as a single complex parameter $\gamma = \mu + i \sigma$ but
two distinct real parameters;
note that this idea was executed by McCullagh \cite{McCullagh1996}.
Thus, our statistical inference for the parameters of Cauchy distributions
will be actually a problem for a single complex parameter,
that is, we guess naturally the location parameter and the scale parameter
at once.

Then we will indeed show that a central limit theorem holds for
the geometric mean (Lemma \ref{lem:CLTforGM}), and then modify it
to fit it to derive our confidence region (Lemma \ref{lem:CLT}) in
the complex plane. It is a standard procedure.
Finally, based on the central limit theorem we obtain the confidence
region for the Cauchy parameters; as a disc in the complex plane
(Theorem \ref{thm:ApproximatingConfidenceDisc}).
We will conclude the paper by showing several numerical examples
for our confidence disc in Section \ref{sec:NumericExamples}.

We wish to express our grateful thanks to an anonymous referee for careful reading of the manuscript and helpful comments which improve the paper.
Especially, Corollaries \ref{cor:confidence-square}
and \ref{cor:confidence-interval} are due to his or her suggestions.

\section{Auxiliary quantities for Cauchy distribution}

\label{sec:auxiliary-results}
In this section, we will gather some basic quantities which will be
used in the following sections.
In the sequel of the paper, we will denote by $C(\mu,\sigma)$ a
Cauchy distribution whose location parameter is $\mu \in \R$
and scale parameter is $\sigma > 0$.
We also denote by $X \sim C(\mu,\sigma)$ a real valued random variable
$X : \Omega \to \R$ when the push forward measure $P_X \equiv PX^{-1}$
on $\R$ is $C(\mu,\sigma)$, where $(\Omega, \F, P)$ is a probability
space that $X$ is defined on. The expectation of $X$ with respect to $P$
will be denoted by $\E[X]$.

For non-integer real number $p \in \R$, we almost surely define $X^p$
by $\exp\{p \log X\}$ as usual when $P(X = 0) = 0$.
Here the logarithmic function $\log$ is defined on a Riemann surface
$\{(z,\theta); z \in \C^*, \theta = \arg z\} \subset \C^* \times \R$, $\C^* := \C \setminus \{0\}$, that is,
$\log: (z, \theta) \mapsto \log |z| + i \theta$ using the usual $\log: \R_+ \to \R$;
$\log$ is a single-valued holomorphic function.
We may simply write $\log z$ as a function on $\C^*$
for $\log(z, \arg z)$ if there is no confusion.
In this case, we regard the logarithmic function has branches,
that is, $\log z = \log |z| + (\arg z + 2k\pi)i$, $0 \leq \arg z < 2\pi$,
$k \in \Z$, is a multivalued function on $\C^*$.

\begin{remark}
Since we have an option to take a branch of $\log$,
the values of $1^p$ and $(-1)^p$ may differ,
for instance if $p = 1/3$,
\begin{equation*}
\begin{array}{c|ccc}
\text{branch} & k = \ldots,-3, 0, 3,\ldots & k = \ldots, -2, 1, 4, \ldots &  k = \ldots, -1, 2, 5, \ldots  \\
\hline
   1^{1/3}    & e^{0\pi i} =  1 & e^{2/3\pi i} = -\frac{1}{2} + \frac{\sqrt{3}}{2}i & e^{4/3\pi i} = -\frac{1}{2} - \frac{\sqrt{3}}{2}i \\
(-1)^{1/3}    & e^{1/3\pi i} = \frac{1}{2} + \frac{\sqrt{3}}{2}i & e^{\pi i} = -1 & e^{5/3\pi i} = \frac{1}{2} - \frac{\sqrt{3}}{2} i\\
\end{array}
\end{equation*}

Note that, even if $x > 0$, $x^p$ may not be a real number but contains
an imaginary part. It does not affect any result while
we keep arbitrariness to choose a branch.

Let us emphasise that both $1^{1/3} = 1$ and $(-1)^{1/3} = -1$
never hold together at the same time.
If we prefer $1^{1/3} = 1$ then $(-1)^{1/3}$ is no more real
and if we prefer $(-1)^{1/3} = -1$ then $1^{1/3}$ becomes
complex. Otherwise, neither $1^{1/3}$ nor $(-1)^{1/3}$ is real.
\end{remark}

Hence, it is natural to handle complex valued random variables.
For such a complex valued random variable $Z$,
the expectation is defined as usual:
$\E[Z] := E[\mathrm{Re}(Z)] + i \E[\mathrm{Im}(Z)]$.
We can also define a variance as a nonnegative real number:
$\V(Z) := \E[ |Z - \E[Z]|^2 ] = \E[|Z|^2] - |\E[Z]|^2$.
Moreover, we may define
pseudo-variance $\PV(Z) := \E[(Z - \E[Z])^2] = \E[Z^2] - (\E[Z])^2$.
If a complex random variable $Z$ satisfies
(1) $\E[Z] = 0$, (2) $\V(Z) < \infty$ and (3) $\E[Z^2] = 0$,
it is called proper.
A proper random variable has a vanishing pseudo-variance.
We set $\gamma := \mu + i\sigma$, and we also write $C(\gamma)$ in place of $C(\mu,\sigma)$.

\begin{proposition}
Let $f : \C \setminus \{0\} \to \C$ be holomorphic and sublinear growth,
that is,
$\lim_{R \to \infty} \frac{1}{R}\sup_{|z| = R}|f(z)| = 0$
and $\lim_{\e \to 0} \e \sup_{|z| = \e}|f(z)| = 0$.
For $X \sim C(\gamma)$,
we see that
$\E[f(X)] = f(\gamma)$.
In particular, the following hold:
\begin{enumerate}
	\item $\E[X^p] = \gamma^p$ for $|p| < 1$; \label{enum:p-moment}
	\item $\E[(\log X)^p] = (\log \gamma)^p$ for all $p \geq 0$.
\end{enumerate}
\end{proposition}
\begin{proof}
Note that the probability density function for Cauchy distribution
is
\begin{equation}\label{eq:Cauchy-density}
p_\gamma(x) =
\frac{\sigma}{\pi} \frac{1}{(x - \mu)^2 + \sigma^2} =
\frac{1}{2\pi i}\left( \frac{1}{x - \gamma} - \frac{1}{x - \bar{\gamma}}\right).
\end{equation}
Since $\sigma > 0$, $1/(z - \bar{\gamma})$ is holomorphic on the upper half plane.
Therefore, the assertion follows immediately from Cauchy's integral formula.
\end{proof}

\begin{corollary}\label{cor:GM-unbiased}
If $X_1, X_2, \ldots, X_N \sim C(\gamma)$, $N \geq 2$, are independent, then
\begin{enumerate}
	\item their geometric mean is an unbiased estimator for
	$\gamma = \mu + i\sigma$, that is,
	$\E\left[ \prod_{j=1}^N X_j^{1/N} \right] = \gamma$;
	\item $\exp\left\{ \E \left[ \frac{1}{N} \sum_{j=1}^N \log X_j \right] \right\} = \gamma$.
\end{enumerate}
\end{corollary}
\begin{remark}\label{rem:geometric-mean}
In this paper, we call $\prod_{j=1}^N X_j^{1/N}$ the
\emph{geometric mean}.
Note that
$\prod_{j=1}^N X_j^{1/N}$ is not equal to $\left(\prod_{j=1}^N X_j\right)^{1/N}$ which may be usually referred to as the geometric mean.
\end{remark}

\begin{corollary}\label{cor:properness}
	If $X \sim C(\gamma)$, then
\begin{enumerate}
	\item pseudo-variances $\PV(X^p) = 0$ $(0 < p < 1/2)$ and $\PV(\log X) = 0$;
	\item $X^p - \gamma^p$ $(0 < p < 1/2)$ and $\log X - \log \gamma$ are
proper complex random variables.
\end{enumerate}
\end{corollary}

\begin{proposition}\label{prop:several-quantities}
If $X \sim C(\gamma)$ and $|p| < 1$, then,
\begin{enumerate}
	\item $\E[ X^p, X > 0] = \frac{\gamma^p - \bar{\gamma}^p}{1 - e^{2p\pi i}} = |\gamma|^p \frac{\sin p(\pi - \arg\gamma)}{\sin p\pi}$; \label{enum:positive-expectation}
	\item $\E[ |X|^p ] = \frac{\gamma^p + (-\bar{\gamma})^p}{1 + e^{p\pi i}}
	= |\gamma|^p \frac{\cos(p(\arg\gamma - \pi/2))}{\cos (p\pi/2)}$;
	\label{enum:abs-p-moment}
	\item $\E[ \log |X| ] = \log |\gamma|$; \label{enum:log-abs-moment}
	\item $\E[ (\log|X|)^2 ] = (\log|\gamma|)^2 + \arg\gamma(\pi - \arg \gamma)$;\label{enum:log-abs-square-moment}
	\item $\V(\log|X|) = \arg \gamma( \pi - \arg \gamma )$;\label{enum:log-abs-variance}
	\item $\V(\log X) = 2 \V(\log|X|) = 2 \arg\gamma ( \pi - \arg\gamma )$. \label{enum:log-variance}
\end{enumerate}
\end{proposition}
\begin{proof}
We will use the probability density of the Cauchy distribution in a form given by \eqref{eq:Cauchy-density}.

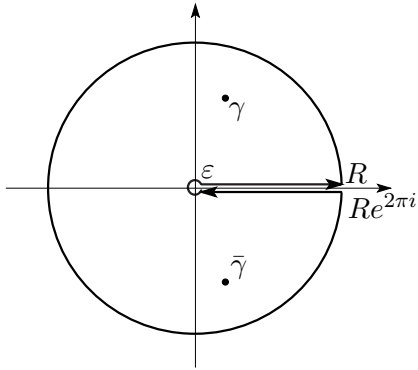
\begin{figure}[htbp]
{\unitlength 0.1in%
\begin{picture}(20.1000,18.9500)(0.2500,-19.2500)%
%
\special{pn 13}%
\special{ar 1002 986 36 38 0.4206633 5.7783239}%
%
\special{pn 13}%
\special{pa 1036 970}%
\special{pa 1758 970}%
\special{fp}%
\special{sh 1}%
\special{pa 1758 970}%
\special{pa 1691 950}%
\special{pa 1705 970}%
\special{pa 1691 990}%
\special{pa 1758 970}%
\special{fp}%
%
\special{pn 4}%
\special{sh 1}%
\special{ar 1160 520 16 16 0 6.2831853}%
%
\special{pn 4}%
\special{sh 1}%
\special{ar 1160 1480 16 16 0 6.2831853}%
\put(11.7800,-5.4200){\makebox(0,0)[lt]{$\gamma$}}%
\put(11.8600,-14.6000){\makebox(0,0)[lb]{$\bar{\gamma}$}}%
\put(10.3400,-9.3800){\makebox(0,0)[lb]{$\varepsilon$}}%
%
\special{pn 4}%
\special{pa 25 990}%
\special{pa 2015 990}%
\special{fp}%
\special{sh 1}%
\special{pa 2015 990}%
\special{pa 1948 970}%
\special{pa 1962 990}%
\special{pa 1948 1010}%
\special{pa 2015 990}%
\special{fp}%
\special{pa 1005 1925}%
\special{pa 1005 30}%
\special{fp}%
\special{sh 1}%
\special{pa 1005 30}%
\special{pa 985 97}%
\special{pa 1005 83}%
\special{pa 1025 97}%
\special{pa 1005 30}%
\special{fp}%
%
\special{pn 13}%
\special{pa 1760 1010}%
\special{pa 1040 1010}%
\special{fp}%
\special{sh 1}%
\special{pa 1040 1010}%
\special{pa 1107 1030}%
\special{pa 1093 1010}%
\special{pa 1107 990}%
\special{pa 1040 1010}%
\special{fp}%
\put(17.7800,-9.6200){\makebox(0,0)[lb]{$R$}}%
\put(17.9200,-10.1600){\makebox(0,0)[lt]{$R e^{2\pi i}$}}%
%
\special{pn 13}%
\special{ar 1002 990 760 760 0.0341972 6.2516169}%
\end{picture}}%
%
	\caption{Integral curve for $\int_0^\infty$}
	\label{fig:integralcurve1}
\end{figure}

First, we will compute an integration along a curve like Figure
\ref{fig:integralcurve1}. Using Cauchy's integral formula and letting $R \to \infty$ and $\e \to 0$, we clearly get
\[
\int_0^\infty x^p p_\gamma(x)\,dx + \int_\infty^0 e^{2p\pi i}x^p p_\gamma(x)\,dx = \gamma^p - \bar{\gamma}^p;
\]
we note that two poles $\gamma$ and $\bar{\gamma}$ stay on the same leaf
but each segment in the left-hand side integrals stays on different leaves.
Thus, we get (\ref{enum:positive-expectation}):
\[
\int_0^\infty x^p p_\gamma(x)\,dx
= \frac{\gamma^p - \bar{\gamma}^p}{1 - e^{2p\pi i}}
= |\gamma|^p \frac{\sin p(\arg\gamma - \pi)}{\sin p\pi}.
\]

Then, to compute (\ref{enum:abs-p-moment}), we first note that
\[\E[ |X|^p ]
= \int_{-\infty}^\infty |x|^p p_\gamma(x)\,dx
= \int_0^\infty x^p ( p_\gamma(x) + p_\gamma(-x))\,dx.
\]
Similarly to (\ref{enum:positive-expectation}), we get,
by Cauchy's integral formula again,
\[
(1 - e^{2p\pi i})\E[|X|^p] = \gamma^p - \bar{\gamma}^p + (-\bar{\gamma})^p - (-\gamma)^p.
\]
Noting that with setting $\theta := \arg \gamma$,
\[
\gamma^p + (-\bar{\gamma})^p = |\gamma| ( e^{ip\theta} + e^{ip(\pi-\theta)})
\ \text{and}\
\bar{\gamma}^p + (-\gamma)^p = e^{ip\pi}(\gamma^p + (-\bar{\gamma})^p),
\]
we finally have (\ref{enum:abs-p-moment}):
\[
\E[|X|^p] = |\gamma|\frac{e^{ip\theta} + e^{ip(\pi-\theta)}}{1 + e^{p\pi i}}.
\]

Now, we note that $g(p) := \E[|X|^p] = \E[e^{p\log|X|}]$ is the
moment generating function of $\log|X|$.
We also note that $z \mapsto \alpha^z$, $\alpha \in \C$
and $z \mapsto 1 + e^{z \pi i}$ are
holomorphic. Therefore, expanding them as
\begin{align*}
\alpha^p &= 1 + (\log \alpha) p + \frac{1}{2!}(\log \alpha)^2p^2 + \frac{1}{3!}(\log \alpha)^3 p^3 + \cdots, \\
\frac{1}{1 + e^{p \pi i}} &= \frac{1}{2} - \frac{1}{4}\pi i p - \frac{1}{48}\pi^3 i p^3 + \cdots,
\end{align*}
allows us the following expansion:
\begin{multline*}
g(p) = \frac{\gamma^p + (-\bar{\gamma})^p}{1 + e^{p\pi i}}
= 1 + \frac{1}{2}( \log \gamma + \log(-\bar{\gamma}) - \pi i)p \\
+ \frac{1}{4}[ ((\log \gamma)^2 + (\log(-\bar{\gamma}))^2) - (\log \gamma + \log(-\bar{\gamma}))\pi i]p^2 + \cdots,
\end{multline*}
so that $\log \gamma + \log(-\bar\gamma) = \log(-|\gamma|^2)
= 2\log|\gamma| + \pi i$ and
$g'(0) = \log|\gamma| = \E[\log|X|]$ provides (\ref{enum:log-abs-moment}).
To show (\ref{enum:log-abs-square-moment}) is routine; letting $\gamma = |\gamma| e^{i\theta}$, $\log \gamma = \log|\gamma| + i\theta$ and
$\log (-\gamma) = \log|\gamma| + i(\theta - \pi)$; $\log|\gamma| \in \R$.
Since $g''(0) = (\log|\gamma|)^2 - \theta(\theta - \pi) = \E[(\log|X|)^2]$,
(\ref{enum:log-abs-variance}) also follows.

To show (\ref{enum:log-variance}),
we first note that $\V(\log X) = \E[|\log X - \log \gamma|^2]
= \E[|\log X|^2] - |\E[\log X]|^2 = \E[|\log X|^2] - |\log \gamma|^2$
from the definition.
Now we may fix a primary branch, that is,
$\log X = \log |X| + 1_{\{X < 0\}} \pi i$.
Hence, $|\log X|^2 = (\log |X|)^2 + 1_{\{X < 0\}} \pi^2$.
Thus, $\E[|\log X|^2] = (\log|\gamma|)^2 + \arg\gamma(\pi - \arg\gamma) + \pi \arg\gamma$.
Since $|\log \gamma|^2 = (\log|\gamma|)^2 + (\arg \gamma)^2$,
we have the conclusion.
\end{proof}

For the geometric mean, we obtain explicit values of the covariances of the real and imaginary parts.

\begin{proposition}
If $n \ge 3$, then,
\begin{multline*}
	 \textup{Var}(\textup{Re}(X_1^{1/n} \cdots X_n^{1/n})) = \textup{Var}(\textup{Im}(X_1^{1/n} \cdots X_n^{1/n})) \\
	 = \frac{r^2}{2} \left(\left(\frac{\cos\left( (2\arg\gamma -\pi)/n \right)}{\cos (\pi/n)}\right)^n  -1 \right)
\end{multline*} 
and
\[ \textup{Cov}\left(\textup{Re}(X_1^{1/n} \cdots X_n^{1/n}), \textup{Im}(X_1^{1/n} \cdots X_n^{1/n})\right) = 0. \]
\end{proposition}

\begin{proof}
Let $\mu := \textup{Re}(\gamma)$ and $\sigma := \textup{Im}(\gamma)$.
By the unbiasedness, we see that
\[ \textup{Var}(\textup{Re}(X_1^{1/n} \cdots X_n^{1/n})) = E\left[ \textup{Re}(X_1^{1/n} \cdots X_n^{1/n})^2\right] - \mu^2\]
and
\[ \textup{Var}(\textup{Im}(X_1^{1/n} \cdots X_n^{1/n})) = E\left[ \textup{Im}(X_1^{1/n} \cdots X_n^{1/n})^2\right] - \sigma^2.\]

Since $X_1, \dots, X_n \in \mathbb R$,
$$\left(X_1^{1/n} \cdots X_n^{1/n}\right)^2 = (X_1^{1/n} X_1^{1/n}) \cdots (X_n^{1/n} X_n^{1/n}) = X_1^{2/n} \cdots X_n^{2/n}.$$

Since
\[ \textup{Re}(X_1^{1/n} \cdots X_n^{1/n})^2 - \textup{Im}(X_1^{1/n} \cdots X_n^{1/n})^2 = \textup{Re}(X_1^{2/n} \cdots X_n^{2/n}), \]
we see that
\[ E[\textup{Re}(X_1^{1/n} \cdots X_n^{1/n})^2] - E[\textup{Im}(X_1^{1/n} \cdots X_n^{1/n})^2] = \mu^2 -\sigma^2.  \]

Hence,
\[ \textup{Var}(\textup{Re}(X_1^{1/n} \cdots X_n^{1/n})) = \textup{Var}(\textup{Im}(X_1^{1/n} \cdots X_n^{1/n})). \]
By this and Proposition \ref{prop:several-quantities} (3),
we  see that $\textup{Var}(\textup{Re}(X_1^{1/n} \cdots X_n^{1/n}))$ and $\textup{Var}(\textup{Im}(X_1^{1/n} \cdots X_n^{1/n}))$ are both equal to $\frac{r^2}{2} \left(\left(\frac{\cos\left( (\pi - 2\arg\gamma)/n \right)}{\cos (\pi/n)}\right)^n  -1 \right)$.

Since
\[ 2\textup{Re}(X_1^{1/n} \cdots X_n^{1/n}) \textup{Im}(X_1^{1/n} \cdots X_n^{1/n}) = \textup{Im}(X_1^{2/n} \cdots X_n^{2/n}), \]
we see that
\[ E\left[\textup{Re}(X_1^{1/n} \cdots X_n^{1/n}) \textup{Im}(X_1^{1/n} \cdots X_n^{1/n})\right] = \mu \sigma.  \qedhere \]
\end{proof}

\section{Central limit theorem and asymptotic confidence disc and square for geometric mean}

To state a central limit theorem for the geometric mean of the
Cauchy random variables, let us begin with recalling that
a complex random variable $Z_{0,1}$ is standard complex normal
if $\mathrm{Re}\,Z_{0,1}$ and $\mathrm{Im}\,Z_{0,1}$ are independent,
and both $\mathrm{Re}\,Z_{0,1}$ and $\mathrm{Im}\,Z_{0,1}$ obey
$N(0, 1/2)$.
Now let us assume that
$X_1, X_2, \ldots$ are independent and identically distributed random
variables with $\E[X_1] = m$ whose real and imaginary parts have
variance $v/2$ and covariance $0$.
Then, it holds that
as $n \to \infty$, $Z_N = [\sum_{j=1}^N X_j - nm] / (\sqrt{Nv})$
converge in law to the standard complex normal random variable
(\cite[Example 2.18]{Vaart1998}).

We already know from Corollary \ref{cor:properness} that,
if $X \sim C(\gamma)$,
$\log X - \log \gamma$ is proper,
that is, the real part and the imaginary part are independent and have
a same distribution with mean 0. And
we also have computed in Proposition \ref{prop:several-quantities} (\ref{enum:log-variance})
that the variance of $\log X$ is $v \equiv 2 \arg \gamma(\pi - \arg \gamma)$.
Therefore, we get the following central limit theorem for the logarithms of
the Cauchy random variables.

\begin{lemma}\label{lem:standardCLT}
Let $X_1, X_2, \ldots \sim C(\gamma)$ be independent.
Then it holds that
\[
\frac{\sqrt{N}}{\sqrt{v}} \left(\frac{1}{N} \sum_{j=1}^N \log X_j - \log \gamma \right) \to Z_{0,1},
\]
as $N \to \infty$ in law.
\end{lemma}

To get a central limit theorem for the geometric mean around $\gamma = \mu + i \sigma$,
we need the following complex version of the delta-method \cite[Theorem 1.12]{Shao2003}, for which the proof is omitted since it is straightforward.

\begin{lemma}
Let $X_1, X_2, \ldots$ and $Y$ be complex random variables satisfying
\[ b_n (X_n - m) \to Y \]
as $n \to \infty$ in law,
where $m \in \C$ and $\{b_n\}$ is a sequence of positive numbers
with $b_n \to \infty$.
Let $f: \C \to \C$ be a holomorphic function on a domain containing $m \in \C$,
then
\[
b_n[ f(X_n) - f(m) ] \to f'(m) Y
\]
as $n \to \infty$ in law.
\end{lemma}

Applying the complex delta method for a complex function $f(z) := e^z$,
we get the following version of the central limit theorem.
\begin{lemma}\label{lem:CLTforGM}
Let $X_1, X_2, \ldots \sim C(\gamma)$ be independent.
Then it holds that
\begin{equation}\label{eq:CTLforGM}
	\frac{\sqrt{N}}{\sqrt{v}\gamma} \left( \prod_{j=1}^N X_j^{1/N} - \gamma \right) \to Z_{0,1},
\end{equation}
as $N \to \infty$ in law.
\end{lemma}

Unfortunately Lemma \ref{lem:CLTforGM} above could not be applied
to get a confidence disc for Cauchy distributed random variables.
We will derive the following variant of the central limit theorem.
\begin{lemma}\label{lem:CLT}
Let $X_1, X_2, \ldots \sim C(\gamma)$ be independent.
Setting
\begin{align}
	P_N &:= \prod_{j=1}^N X_j^{1/N}, \notag \\ \intertext{and}
	V_N &:= \frac{1}{N-1} \sum_{j=1}^N |\log X_j|^2 - \left|\frac{1}{N}\sum_{j=1}^N \log X_j \right|^2, \label{eq:sample-variance}\end{align}
which is an unbiased and consistent estimator for
$v := \V(\log X_1) = 2 \arg \gamma(\pi - \arg \gamma)$,
we have
\[
\frac{\sqrt{N}}{\sqrt{V_N} P_N} \left( P_N - \gamma \right) \to Z_{0,1},
\]
as $N \to \infty$ in law.
\end{lemma}
\begin{proof}
Since $V_N$ converges to $2 \arg \gamma ( \pi - \arg \gamma)$
and $P_N$ converges to $\gamma$ as $N \to \infty$
in probability, Lemma \ref{lem:CLTforGM} implies the conclusion.
\end{proof}

Now we can state a formula for our asymptotic confidence disc.

\begin{theorem}\label{thm:ApproximatingConfidenceDisc}
Let $X_1, X_2, \ldots, X_N$ be independent Cauchy random variables
with location $\mu \in \R$ and scale $\sigma > 0$.
Set $P_N$ and $V_N$ as in Lemma \ref{lem:CLT}.
Then, we have
\begin{equation}\label{eq:confidence-disc}
\lim_{N \to \infty} P\left(
 \gamma \in B\left( P_N, \frac{\sqrt{V_N}}{\sqrt{N}} R_\alpha |P_N| \right) \right) = 1 - \alpha,
\end{equation}
for $0 < \alpha \leq 1$,
where $R_\alpha = \sqrt{- \log \alpha}$ and
$B(p,r)$ denotes a disc in $\C$ with centre $p$ and radius $r$.
\end{theorem}



\begin{proof}
We first note that
$P(|Z_{0,1}| \leq \sqrt{-\log \alpha}) = 1 - \alpha$.
Combining it with Lemma \ref{lem:CLT} asserts that
\[ \lim_{N \to \infty}
P\left( \left| \frac{\sqrt{N}}{\sqrt{V_N} P_N} (P_N - \gamma) \right| \leq R_\alpha \right)
= 1 - \alpha,
\]
which immediately leads to the conclusion.
\end{proof}

A similar argument enables us to derive an
asymptotic confidence square that is slightly smaller
than the disc.

\begin{corollary}\label{cor:confidence-square}
	We denote by $Q(p,r)$ a square in the complex plane whose centre is $r$
	and the length of one side is $2r$, namely,
	$Q(p,r) = [p-r, p+r] \times i[p-r, p+r]$.
	Let $\rho_\alpha$ be the upper $\alpha$-quantile of standard real normal distribution, that is, $\int_{-\infty}^{\rho_\alpha} \frac{1}{\sqrt{2\pi}} \exp\{ -x^2/2 \}\,dx = 1 - \alpha$.
	Under same settings of Theorem \ref{thm:ApproximatingConfidenceDisc},
	and $\beta = \frac{1}{2}(1 - \sqrt{1 - \alpha})$,	we have
	\[\lim_{N \to \infty}
	P\left(
	\gamma \in Q\left(P_N, \frac{\sqrt{V_N}}{\sqrt{2N}}|P_N| \rho_{\beta}\right)
	\right) = 1 - \alpha
	\]
\end{corollary}

\begin{proof}
	Recall that $Z_{0,1}$ is rotation invariant in law, that is,
	for any complex number $w$ with $|w| = 1$,
	$Z_{0,1}$ and $w Z_{0,1}$ share a same distribution.
	This fact guarantees that we may take an absolute value in
	the coefficient of \eqref{eq:CTLforGM}, that is, we have
	\[
		\frac{\sqrt{N}}{\sqrt{v}|\gamma|} \left( P_N - \gamma \right) \to Z_{0,1}
	\]
	as $N \to \infty$ in law.
	Therefore, a same argument with Lemma \ref{lem:CLT} leads
	\[\lim_{N \to \infty}
	P\left(
		\gamma \in Q\left(P_N, \frac{\sqrt{V_N}|P_N|}{\sqrt{2N}} \rho_{\beta}\right)
		\right)
		= P\left( Z_{0,1} \in
		Q\left(0, \rho_{\beta}/\sqrt{2}\right)
		\right).
	\]
	Since $\mathrm{Re}(Z_{0,1})$ and $\mathrm{Im}(Z_{0,1})$ are independent,
	and obeys $N(0,1/2)$, the right-hand side equals to $(1 - 2\beta)^2 = 1 - \alpha$.
\end{proof}

We can apply the idea of Corollary \ref{cor:confidence-square}
to get confidence intervals for $\mu$ and $\sigma$, respectively.

\begin{corollary}\label{cor:confidence-interval}
	Under the same settings of Corollary \ref{cor:confidence-square},
	we have
	\begin{align*}
		\lim_{N\to\infty} P\left( \mathrm{Re}(P_N) - \frac{\sqrt
		{V_N}|P_N|}{\sqrt{2N}} \rho_{\alpha/2} \leq
		\mu \leq
		\mathrm{Re}(P_N) + \frac{\sqrt
		{V_N}|P_N|}{\sqrt{2N}} \rho_{\alpha/2} \right) = 1 - \alpha \\
		\lim_{N\to\infty} P\left( \mathrm{Im}(P_N) - \frac{\sqrt
		{V_N}|P_N|}{\sqrt{2N}} \rho_{\alpha/2} \leq
		\sigma \leq
		\mathrm{Im}(P_N) + \frac{\sqrt
		{V_N}|P_N|}{\sqrt{2N}} \rho_{\alpha/2} \right) = 1 - \alpha
	\end{align*}
\end{corollary}
\begin{proof}
	Since we have $P(Z_{0,1} \in [-\rho_{\alpha/2}/\sqrt{2}, \rho_{\alpha/2}/\sqrt{2}] \times i \mathbb{R})
	= 1 - \alpha$,
	the results follow from the same computation with Corollary \ref{cor:confidence-square}.
\end{proof}


\section{Numerical examples}\label{sec:NumericExamples}
In this section we will show several numerical examples
of Theorem \ref{thm:ApproximatingConfidenceDisc}.

\subsection{\mathversion{bold}$\mu = 0, \sigma = 1, N=1000, \alpha =0.05$, 10 trials}
\begin{figure}[htbp]
	\begin{center}
	\setbox0=\hbox{\includegraphics[width=.4\linewidth]{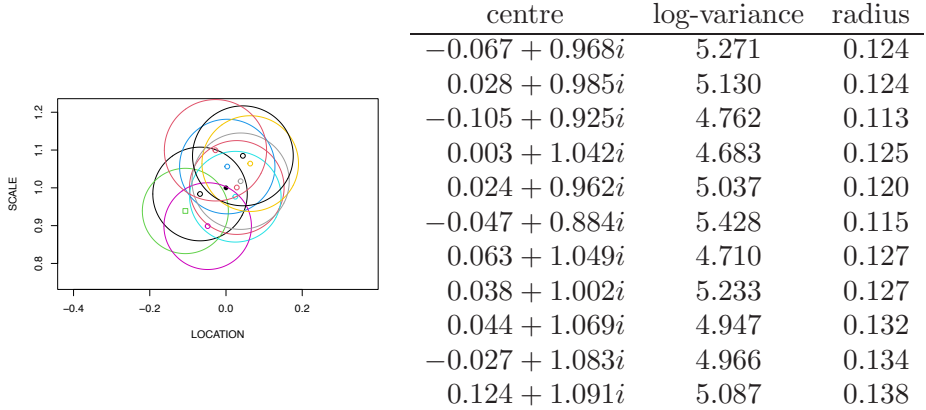}}
	\raisebox{-.5\ht0}{\copy0}
	\resizebox{.55\linewidth}{!}{
	\begin{tabular}{rcrc}
	\multicolumn{1}{c}{centre} & log-variance & radius \\
	\hline
	$-0.067+0.968i$ & $5.271$ & $0.124$ \\
	$0.028+0.985i$  & $5.130$ & $0.124$ \\
	$-0.105+0.925i$ & $4.762$ & $0.113$ \\
	$0.003+1.042i$  & $4.683$ & $0.125$ \\
	$0.024+0.962i$  & $5.037$ & $0.120$ \\
	$-0.047+0.884i$ & $5.428$ & $0.115$ \\
	$0.063+1.049i$  & $4.710$ & $0.127$ \\
	$0.038+1.002i$  & $5.233$ & $0.127$ \\
	$0.044+1.069i$  & $4.947$ & $0.132$ \\
	$-0.027+1.083i$ & $4.966$ & $0.134$ \\
	$ 0.124+1.091i$ & $5.087$ & $0.138$
%
%
	\end{tabular}}
	\end{center}
	\caption{$\mu = 0,\ \sigma = 1,\ N=1000,\ \alpha=0.05$, 10 trials}
	\label{fig:N=1000,T=10,A=0.05}
\end{figure}
Figure \ref{fig:N=1000,T=10,A=0.05} shows 10 confidence discs;
each of the samples  has of size 1,000.
Namely, we generated samples of
$X^1_1, X^1_2,\ldots,X^1_{1000}, X^2_1, X^2_2, \ldots, X^2_{1000}, \ldots,
X^{10}_1, X^{10}_2, \ldots, X^{10}_{1000} \sim C(0,1)$.
Then we computed the geometric mean $P^i_{1000} := \prod_{j=1}^{1000} (X^i_{j})^{1/1000}$ by the formula $P^i_{1000} = \exp\{ \frac{1}{1000} \sum_{j=1}^{1000} \log X^i_j \}$,
and sample variance $V^i_{1000}$ by \eqref{eq:sample-variance}.
We drew 10 discs $B^1, \ldots, B^{10}$
using these quantities and \eqref{eq:confidence-disc}.
A symbol $\bullet$ indicates the true value $(\mu, \sigma) = (0,1)$.
Symbols $\circ$ point the centre of the disc containing the true value,
while $\square$ is the point that fails.

\subsection{\mathversion{bold}$\mu = 5, \sigma = 1, N=1000, \alpha =0.05$, 10 trials}
Figure \ref{fig:N=1000,T=10,MU=5,SIGMA=1,A=0.05} also shows
the same size of trials and samples, but
$\mu = 5$ and $\sigma = 1$ for larger true values;
we fix $\sigma = 1$ to compare with Figure \ref{fig:N=1000,T=10,A=0.05} easily.
Note that, since $\arg(5+i)=\tan^{-1}\frac{1}{5} \sim 0.197$, $\V(\log X) \sim 1.162$. Note that $\V(\log X)$ is small if $|\mu|$ is large but
the radius of the confidence discs increases as $|\mu|$ increases.

\begin{figure}[htbp]
	\begin{center}
	\setbox0=\hbox{\includegraphics[width=.4\linewidth]{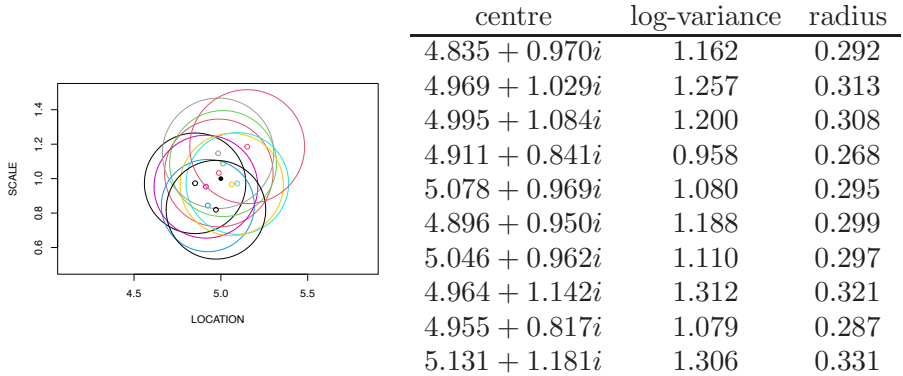}}
	\raisebox{-.5\ht0}{\copy0}
	\resizebox{.55\linewidth}{!}{
	\begin{tabular}{cccc}
	centre & log-variance & radius \\
	\hline
	$4.835+0.970i$ & $1.162$ & $0.292$ \\
	$4.969+1.029i$ & $1.257$ & $0.313$ \\
	$4.995+1.084i$ & $1.200$ & $0.308$ \\
	$4.911+0.841i$ & $0.958$ & $0.268$ \\
	$5.078+0.969i$ & $1.080$ & $0.295$ \\
	$4.896+0.950i$ & $1.188$ & $0.299$ \\
	$5.046+0.962i$ & $1.110$ & $0.297$ \\
	$4.964+1.142i$ & $1.312$ & $0.321$ \\
	$4.955+0.817i$ & $1.079$ & $0.287$ \\
	$5.131+1.181i$ & $1.306$ & $0.331$
%
%
	\end{tabular}}
	\end{center}
	\caption{$\mu = 5, \sigma = 1$, $N=1000, \alpha=0.05$, 10 trials}
	\label{fig:N=1000,T=10,MU=5,SIGMA=1,A=0.05}
\end{figure}

\subsection{\mathversion{bold}$\mu = 0, \sigma = 1, N=30, \alpha = 0.05$, 1000 trials}
Since the formula \eqref{eq:confidence-disc} to compute our confidence disc
is based on a central limit theorem (Theorem \ref{lem:CLT}).
Therefore, we may fail to guess the parameters if $N$ is small.
\begin{figure}[htbp]
	\includegraphics[width=.4\linewidth]{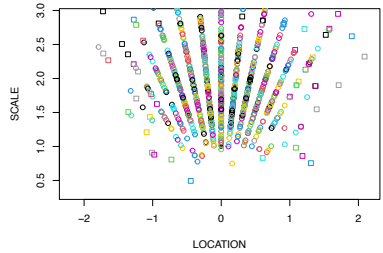}
	\caption{$\mu = 0$, $\sigma = 1$, $N=30$, $\alpha = 0.05$, 1000 trials}
	\label{fig:N=30,TRIAL=1000}
\end{figure}
Figure \ref{fig:N=30,TRIAL=1000} (omitting to draw circles) shows the case $N = 30$ with $(\mu, \sigma) = (0,1)$.
In this example, computed $\alpha = 0.05$, only 925 trials of 1,000 contain the true value.

\subsection{\mathversion{bold}Large $\mu$ and small $\sigma$}\label{subsec:sbstract-median}
When $|\mu|$ is much larger than $\sigma$, all data may have the same sign.
In such a case, there is a possibility that
the geometric mean has the vanishing imaginary part so that
we shall guess $\sigma$ as zero.
\begin{figure}[htbp]
	\begin{center}
	\setbox0=\hbox{\includegraphics[width=.4\linewidth]{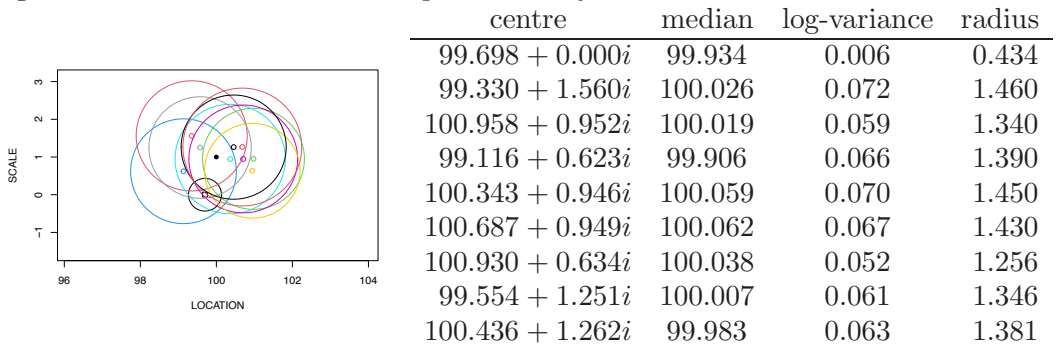}}
	\raisebox{-.5\ht0}{\copy0}
	\resizebox{.55\linewidth}{!}{
	\begin{tabular}{rccrc}
	\multicolumn{1}{c}{centre} & median & log-variance & radius \\
	\hline
	$ 99.698+0.000i$ & $ 99.934$ & $0.006$ & $0.434$ \\
	$ 99.330+1.560i$ & $100.026$ & $0.072$ & $1.460$ \\
	$100.958+0.952i$ & $100.019$ & $0.059$ & $1.340$ \\
	$ 99.116+0.623i$ & $ 99.906$ & $0.066$ & $1.390$ \\
	$100.343+0.946i$ & $100.059$ & $0.070$ & $1.450$ \\
	$100.687+0.949i$ & $100.062$ & $0.067$ & $1.430$ \\
	$100.930+0.634i$ & $100.038$ & $0.052$ & $1.256$ \\
	$ 99.554+1.251i$ & $100.007$ & $0.061$ & $1.346$ \\
	$100.436+1.262i$ & $ 99.983$ & $0.063$ & $1.381$
%
%
	\end{tabular}}
	\end{center}
	\caption{$\mu = 100, \sigma = 1$, $N=1000, \alpha=0.05$, 10 trials}
	\label{fig:N=1000,T=10,MU=100,SIGMA=1,A=0.05}
\end{figure}
Figure \ref{fig:N=1000,T=10,MU=100,SIGMA=1,A=0.05} shows such a case.
One practical way to avoid the situation is to subtract a constant
from each datum.
It is clear that, putting $Y_i := X_i - c$, an estimate using $Y_i$'s plus
$c$ is equal to that of $X_i$'s.
We also emphasise that subtracting a constant may decrease the absolute value
of the geometric mean so that our estimated discs may be smaller, i.e.,
we may improve the estimation.
 However, it is impossible to get
a suitable constant to make the data consist of both positive and negative
values.
Therefore, practically speaking, it may be a candidate to subtract
the median
$\med(X_1, \ldots, X_N)$ from all $X_i$'s. Then the data will consist of same
numbers of positive and negative ones.
\begin{figure}[htbp]
	\begin{center}
	\setbox0=\hbox{\includegraphics[width=.4\linewidth]{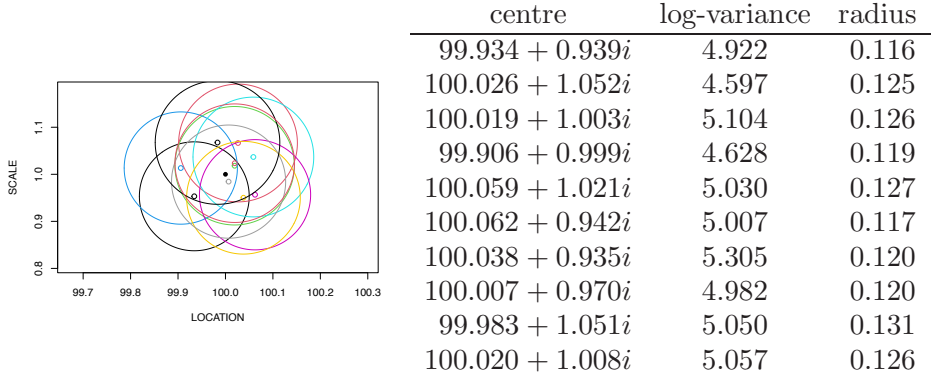}}
	\raisebox{-.5\ht0}{\copy0}
	\resizebox{.55\linewidth}{!}{
	\begin{tabular}{rcrc}
	\multicolumn{1}{c}{centre} & log-variance & radius \\
	\hline
	$ 99.934+0.939i$ & $4.922$ & $0.116$ \\
	$100.026+1.052i$ & $4.597$ & $0.125$ \\
	$100.019+1.003i$ & $5.104$ & $0.126$ \\
	$ 99.906+0.999i$ & $4.628$ & $0.119$ \\
	$100.059+1.021i$ & $5.030$ & $0.127$ \\
	$100.062+0.942i$ & $5.007$ & $0.117$ \\
	$100.038+0.935i$ & $5.305$ & $0.120$ \\
	$100.007+0.970i$ & $4.982$ & $0.120$ \\
	$ 99.983+1.051i$ & $5.050$ & $0.131$ \\
	$100.020+1.008i$ & $5.057$ & $0.126$ \\
	%
	%
	\end{tabular}}
	\end{center}
	\caption{Same data with Figure \ref{fig:N=1000,T=10,MU=100,SIGMA=1,A=0.05} with subtracting the median}
	\label{fig:N=1000,T=10,MU=100,SIGMA=1,A=0.05,SUBTRACT}
\end{figure}
Figure \ref{fig:N=1000,T=10,MU=100,SIGMA=1,A=0.05,SUBTRACT} shows the
effect of such manipulations.
Let us, however, mention that $X_i - \med(X_1,\ldots,X_N)$ are
neither Cauchy nor independent though the central limit theorem may still hold.
It seems difficult for us to get a suitable random variable to subtract
from each datum.
Figure \ref{fig:N=1000,T=10,MU=100,SIGMA=1,A=0.05,SUBTRACT}
is shown just for convenience;
we will not provide any detailed mathematically reasonable explanation
for this sort of manipulations.

\subsection{Outlier and Gaussian estimation}

Here we deal with one-dimensional location-scale families and consider the estimation of the location when the scale is not known.
It is common to assume the data are Gaussian distributed
if the shape of the frequency forms a bell shape.
But when the data have some outliers, using
$t$-distributions may give an inappropriate confidence interval.
But it should be noted that the Cauchy distribution often gives such
``outliers'' due to its heavy tail of the density
so that we may expect that an inference that is applicable to the Cauchy
distribution is more robust.

Here we show a numerical example, which consists of 10 samples;
each of them contains 100 values from a standard normal distribution.
Then we replace each 100th datum with ``5'' to mimic an outlier.
Tables \ref{tab:outlier-tdistribution} and \ref{tab:outlier-theorem42} below exhibit these differences.
The left parts of those tables are the original data, and the right parts are the modified data containing the outliers.

Table \ref{tab:outlier-tdistribution} shows how such contamination varies the confidence intervals
when we use the $t$-distribution with the degree of freedom 99,
while Table \ref{tab:outlier-theorem42} shows those by using Corollary \ref{cor:confidence-interval}.
Here the AM and GM denote the arithmetic mean and the real part of the geometric mean,
respectively.

\begin{table}[htbp]
	\[
	\begin{array}{c|rrr||rrr}
	\text{sample} & \multicolumn{1}{c}{\text{AM}}
								& \multicolumn{1}{c}{\text{radius}}
								& \multicolumn{1}{c}{\text{interval}}
								& \multicolumn{1}{c}{\text{AM}}
								& \multicolumn{1}{c}{\text{radius}}
								& \multicolumn{1}{c}{\text{interval}} \\
	\hline
	1 & -0.049 & 0.200 & [-0.249, 0.150] & -0.002 & 0.223 & [-0.225, 0.221] \\
	2 & 0.061  & 0.165 & [-0.103, 0.226] & 0.116  & 0.191 & [-0.075, 0.308] \\
	3 & 0.075  & 0.197 & [-0.122, 0.271] & 0.127  & 0.219 & [-0.093, 0.346] \\
	4 & -0.076 & 0.185 & [-0.261, 0.109] & -0.010 & 0.208 & [-0.219, 0.198] \\
	5 & -0.001 & 0.219 & [-0.221, 0.218] & 0.029  & 0.237 & [-0.209, 0.266] \\
	6 & 0.042  & 0.222 & [-0.180, 0.264] & 0.105  & 0.241 & [-0.136, 0.346] \\
	7 & -0.124 & 0.199 & [-0.323, 0.074] & -0.062 & 0.222 & [-0.285, 0.160] \\
	8 & 0.094  & 0.228 & [-0.134, 0.322] & 0.147  & 0.248 & [-0.101, 0.395] \\
	9 & -0.091 & 0.181 & [-0.272, 0.091] & -0.032 & 0.207 & [-0.238, 0.175] \\
	10 & 0.180  & 0.203 & [-0.024, 0.383] & 0.243  & 0.223 & [0.020, 0.465]
	\end{array}
	\]
	\caption{outlier's effect for $t$-distribution}
	\label{tab:outlier-tdistribution}
\end{table}

\begin{table}[htbp]
	\[
	\begin{array}{c|rrr||rrr}
	\text{sample}
			& \multicolumn{1}{c}{\text{GM}}
			& \multicolumn{1}{c}{\text{radius}}
			& \multicolumn{1}{c}{\text{interval}}
			& \multicolumn{1}{c}{\text{GM}}
			& \multicolumn{1}{c}{\text{radius}}
			& \multicolumn{1}{c}{\text{interval}} \\
	\hline
	1 & 0.017  & 0.143 & [-0.126,0.160] & 0.017  & 0.148 & [-0.130,0.165] \\
	2 & 0.025  & 0.110 & [-0.085,0.135] & 0.038  & 0.113 & [-0.075,0.152] \\
	3 & 0.032  & 0.136 & [-0.104,0.168] & 0.050  & 0.141 & [-0.091,0.191] \\
	4 & -0.131 & 0.138 & [-0.269,0.006] & -0.117 & 0.140 & [-0.257,0.024] \\
	5 & -0.061 & 0.163 & [-0.224,0.102] & -0.061 & 0.165 & [-0.227,0.104] \\
	6 & 0.017  & 0.150 & [-0.133,0.167] & 0.034  & 0.153 & [-0.119,0.187] \\
	7 & 0.016  & 0.137 & [-0.121,0.153] & 0.032  & 0.140 & [-0.108,0.172] \\
	8 & 0.099  & 0.162 & [-0.063,0.261] & 0.122  & 0.167 & [-0.046,0.289] \\
	9 & -0.155 & 0.123 & [-0.278,-0.032] & -0.144 & 0.126 & [-0.270,-0.018] \\
	10 & 0.206  & 0.162 & [0.044,0.368] & 0.227  & 0.165 & [0.063,0.392]
	\end{array}
	\]
	\caption{outlier's effect for Corollary \ref{cor:confidence-interval}}
	\label{tab:outlier-theorem42}
\end{table}

It can be seen that the confidence intervals
computed using Corollary \ref{cor:confidence-interval}
are not so affected by outliers.

\subsection{Remark about subtraction}

As we see in subsection \ref{subsec:sbstract-median} above,
the manipulation of subtracting the median works well when the location is far from the origin.
However, subtraction entails a delicate problem.

Let $\mu$ and $\sigma$ be the real and imaginary parts of $\gamma$.
Consider
$$Z_N := Y_N + \prod_{i=1}^{N} (X_{i} - Y_{N})^{1/N}.$$
There is a sequence of random variable $Y_N$ such that $Y_N \to \mu$ a.s. and $Z_N$ in the above does {\it not} converge to $\gamma$ in probability as $N \to \infty$.

This occurs if we let $Y_N := X^{(\lfloor N/2 \rfloor + 1)}$, where $\{X_1, \cdots, X_N \} = \{X^{(1)} \le \cdots \le X^{(N)}\}$ and $\lfloor r \rfloor$ denotes the integer part of $r \in \mathbb{R}$.
In that case we see that
$Z_N = Y_N$
and in particular $\textup{Im}(Z_N) = 0$.
$Y_N$ is the median of $\{X_1, \cdots, X_N \}$ if $N$ is {\it odd}.
In the above subsection, we deal with the case that $N$ is {\it even}.

Let $I$ be a closed interval containing the location $\mu$ in its interior.
Then,
$$\sup_{\theta \in I} \left| \theta +  \prod_{j=1}^{N} (X_{j} - \theta)^{1/N}  - \gamma  \right| \ge \sigma. $$

Since $P(Y_N \in I) \to 1, \ N \to \infty$,
we see that
$$P\left(\sup_{\theta \in I} \left| \theta +  \prod_{j=1}^{N} (X_{j} - \theta)^{1/N}  - \gamma \right| \ge \sigma\right) \to 1, \ \ N \to \infty.$$

We also have the following.

\begin{proposition}\label{prop:uniformintegrable}
Let $I$ be a closed interval.
Then,
$$\left(\sup_{\theta \in I} \left| \theta +  \prod_{j=1}^{N} (X_{j} - \theta)^{1/N}  - \gamma \right|\right)_N$$
is uniformly integrable.
\end{proposition}

\begin{proof}
Let $M := \max\{|\theta - \mu| : \theta \in I\}$.
Then, for every $\theta \in I$,
\[ \left| \theta +  \prod_{j=1}^{N} (X_{j} - \theta)^{1/N} - \gamma \right| \le |\gamma| + M +  \prod_{j=1}^{N} ( |X_{j} - \mu| + M)^{1/N}. \]
Hence,
\[ E\left[ \sup_{\theta \in I} \left| \theta +  \prod_{j=1}^{N} (X_{j} - \theta)^{1/N} \right| \right] \le M + E\left[( |X_{1} - \mu| + M)^{1/N}\right]^N. \]
Since
\[ \lim_{N \to \infty} E\left[( |X_{1} - \mu| + M)^{1/N}\right]^N = \exp\left( E\left[ \log\left(|X_{1} - \mu| + M\right) \right]\right), \]
we see that
\[ \sup_N E\left[( |X_{1} - \mu| + M)^{1/N}\right]^N < +\infty. \]
\end{proof}

On the other hand,
 if we let $Y_N := (X^{(\lfloor N/2 \rfloor + 1)} + X^{(\lfloor N/2 \rfloor)})/2$,
 then, for even number $N$, $Y_N$ is the median of $\{X_1, \cdots, X_N\}$, and by numerical computations, $Z_N$ approximates $\gamma$ well, and we conjecture that
 $\lim_{N \to \infty} Z_{N} = \gamma$, \textup{ a.s.}
However, we do not have any proofs of it.
We do not see why the case that $Y_N = X^{(\lfloor N/2 \rfloor + 1)}$ is bad, but, on the other hand,  the case that $Y_N = (X^{(\lfloor N/2 \rfloor + 1)} + X^{(\lfloor N/2 \rfloor)})/2$ is good,
although $X^{(\lfloor N/2 \rfloor + 1)}$ and $(X^{(\lfloor N/2 \rfloor + 1)} + X^{(\lfloor N/2 \rfloor)})/2$ are close to each other with high probability if $N$ is large.

One way to resolve this delicate issue is to consider the estimator
\[ \theta - \epsilon i +  \prod_{j=1}^{N} (X_{j} - \theta + \epsilon i)^{1/N}  \]
for some $\epsilon > 0$.
This is also an unbiased estimator of $\gamma$.
This is easier to handle due to the following.

\begin{proposition}
Let $I$ be a closed interval.
Then,
\[ \sup_{\theta \in I} \left| \theta - \epsilon i +  \prod_{j=1}^{N} (X_{j} - \theta + \epsilon i)^{1/N}  - \gamma \right| \]
converges to $0$ a.s. and in $L^1$ as $N \to \infty$.
\end{proposition}

\begin{proof}
The a.s. convergence follows from the uniform law of large numbers for $\left(\log(X_N - \theta + \epsilon i) - \log\left(\mu - \theta + (\sigma + \epsilon)i\right) \right)_N$.

In the same manner as in the proof of Proposition \ref{prop:uniformintegrable},
we can show that
$$\left(\sup_{\theta \in I} \left| \theta - \epsilon i +  \prod_{j=1}^{N} (X_{j} - \theta + \epsilon i)^{1/N}  -\gamma \right|\right)_N$$
is uniformly integrable.
Thus, we see the $L^1$ convergence.
\end{proof}

There is also a problem in this resolution because there are many possibilities of $\epsilon > 0$, and we do not see which $\epsilon > 0$ should be taken.
We do not delve into this issue in this paper. \\

{\it Acknowledgements} \ The authors wish to give their gratitude to Dr Thomas Simon for references.
The second and third authors were supported by JSPS KAKENHI 19K14549 and 16K05196 respectively.

\bibliographystyle{amsplain}
\bibliography{Cauchy3}

\providecommand{\bysame}{\leavevmode\hbox to3em{\hrulefill}\thinspace}
\providecommand{\MR}{\relax\ifhmode\unskip\space\fi MR }
\providecommand{\MRhref}[2]{%
  \href{http://www.ams.org/mathscinet-getitem?mr=#1}{#2}
}
\providecommand{\href}[2]{#2}
\begin{thebibliography}{10}

\bibitem{Akaoka2020a}
Yuichi Akaoka, \emph{Parameter estimation using complex valued moments for
  {Cauchy} distributions}, Master's thesis, Department of mathematics, Shinshu
  University, January 2020.

\bibitem{Akaoka2021-2}
Yuichi Akaoka, Kazuki Okamura, and Yoshiki Otobe, \emph{Bahadur efficiency of
  the maximum likelihood estimator and one-step estimator for quasi-arithmetic
  means of the {C}auchy distribution}, Annals of the Institute of Statistical
  Mathematics (2022).

\bibitem{Akaoka2021-1}
\bysame, \emph{Limit theorems for quasi-arithmetic means of random variables
  with applications to point estimations for the {C}auchy distribution},
  Brazilian Journal of Probability and Statistics (2022), to appear.

\bibitem{CohenFreue2007}
Gabriela~V. Cohen~Freue, \emph{The {P}itman estimator of the {C}auchy location
  parameter}, Journal of Statistical Planning and Inference \textbf{137}
  (2007), no.~6, 1900--1913.

\bibitem{Ferguson1978}
Thomas~S. Ferguson, \emph{Maximum likelihood estimates of the parameters of the
  {C}auchy distribution for samples of size 3 and 4}, Journal of the American
  Statistical Association \textbf{73} (1978), no.~361, 211--213.

\bibitem{Haas1970}
Gerald Haas, Lee Bain, and Charles Antle, \emph{Inferences for the {Cauchy}
  distribution based on maximum likelihood estimators}, Biometrika \textbf{57}
  (1970), no.~2, 403--408.

\bibitem{Haas1969}
Gerald~Nicholas Haas, \emph{Statistical inferences for the {C}auchy
  distribution based on maximum likelihood estimators}, Doctoral dissertations.
  2274., University of Missouri--Rolla, 1969.

\bibitem{Hinkley1978}
D.~V. Hinkley, \emph{Likelihood inference about location and scale parameters},
  Biometrika \textbf{65} (1978), no.~2, 253--261.

\bibitem{Kravchuk2012}
O.~Y. Kravchuk and P.~K. Pollett, \emph{{Hodges}--{Lehmann} scale estimator for
  {Cauchy} distribution}, Communications in Statistics - Theory and Methods
  \textbf{41} (2012), no.~20, 3621--3632.

\bibitem{Lawless1972}
J.~F. Lawless, \emph{Conditional confidence interval procedures for the
  location and scale parameters of the {Cauchy} and logistic distributions},
  Biometrika \textbf{59} (1972), no.~2, 377--386.

\bibitem{McCullagh1993}
Peter McCullagh, \emph{On the distribution of the {C}auchy maximum-likelihood
  estimator}, Proceedings of the Royal Society. London. Series A \textbf{440}
  (1993), 475--479.

\bibitem{McCullagh1996}
\bysame, \emph{{M}{\"o}bius transformation and {C}auchy parameter estimation},
  The Annals of Statistics \textbf{24} (1996), no.~2, 787--808.

\bibitem{Okamura2021}
Kazuki Okamura and Yoshiki Otobe, \emph{Characterizations of the maximum
  likelihood estimator of the {C}auchy distribution}, Lobachevskii Journal of
  Mathematics (2021), to appear.

\bibitem{Rothenberg1964}
Thomas~J. Rothenberg, Franklin~M. Fisher, and C.~B. Tilanus, \emph{A note on
  estimation from a {C}auchy sample}, Journal of the American Statistical
  Association \textbf{59} (1964), no.~306, 460--463.

\bibitem{Shao2003}
Jun Shao, \emph{Mathematical statistics}, 2\textsuperscript{nd} ed., Springer
  New York, 2003.

\bibitem{Vaart1998}
A.~W. van~der Vaart, \emph{Asymptotic statistics}, Cambridge Series in
  Statistical and Probabilistic Mathematics, Cambridge University Press,
  October 1998.

\bibitem{Vrbik2013}
Jan {Vrbik}, \emph{{Accurate confidence regions based on MLEs}}, {Adv. Appl.
  Stat.} \textbf{32} (2013), no.~1, 33--56.

\bibitem{Zolotarev1986}
V~M Zolotarev, \emph{One-dimensional stable distributions}, American
  Mathematical Society, 1986.

\end{thebibliography}
\end{document}